\documentclass[11pt]{article}
\usepackage{amsthm}
\usepackage{amssymb}
\usepackage{amsmath}
\usepackage{amsfonts}
\usepackage{float}
\usepackage{xcolor}

\usepackage{hyperref}

\newcounter{contador}
\newcounter{teoA}

\newtheorem{teoa}[teoA]{Theorem}

\newtheorem{propo}[contador]{Proposition}
\newtheorem{teo}[contador]{Theorem}
\newtheorem{lem}[contador]{Lemma}

\theoremstyle{remark}

\newcounter{ex}

\newcommand{\R}{{\mathbb R}}

\newcommand{\N}{{\mathbb N}}

\newcommand{\Z}{{\mathbb Z}}

\title{On two families of iterative methods without memory}

\author{Anna Cima$^{(1)}$, Armengol Gasull$^{(1)}$, V\'{\i}ctor Ma\~{n}osa$^{(2)}$ and Francesc Ma\~{n}osas$^{(1)}$
	\\*[.1truecm]
	{\small \textsl{$^{(1)}$ Departament de Matem\`{a}tiques, Facultat
			de Ci\`{e}ncies,}}
	\\*[-.25truecm] {\small \textsl{Universitat Aut\`{o}noma de Barcelona,}}
	\\*[-.25truecm] {\small \textsl{08193 Bellaterra, Barcelona,
			Spain}}
	\\*[-.25truecm] {\small \textsl{anna.cima@uab.cat,
			armengol.gasull@uab.cat, francesc.manosas@uab.cat}}\\
	\\*[-.25truecm] {\small \textsl{$^{(2)}$ Departament de Matem\`{a}tiques (MAT),}}
	\\*[-.25truecm] {\small \textsl{Institut de Matem\`{a}tiques de la UPC-BarcelonaTech (IMTech),}}           
	\\*[-.25truecm] {\small \textsl{Universitat Polit\`{e}cnica de Catalunya (UPC)}}
	\\*[-.25truecm] {\small \textsl{Colom 11, 08222 Terrassa, Spain}}
	\\*[-.25truecm] {\small \textsl{victor.manosa@upc.edu}}}

\date{\today}

\begin{document}

\maketitle
\begin{abstract}  
We study two natural families of methods of order  $n\ge 2$ that are useful for solving numerically one variable equations	 $f(x)=0.$  The first family consists on the methods that depend on $x,f(x)$ and its successive derivatives up to $f^{(n-1)}(x)$ and the second family comprises methods that depend on $x,g(x)$  until  $g^{\circ n}(x),$ where $g^{\circ m}(x)=g(g^{\circ (m-1)}(x))$ and $g(x)=f(x)+x$. The first family includes the well-known Newton, Chebyshev, and Halley methods, while the second one contains the Steffensen method. Although the results for the first type of methods are well known and classical, we provide new, simple, detailed, and self-contained proofs.

\end{abstract}

%
%

\noindent {\sl  Mathematics Subject Classification:} 65H05, 58C30. 

\noindent {\sl Keywords:} Newton method, Chebyshev method, Halley method, Steffensen method, iterative methods, efficiency, zeros of real functions.

\section{Introduction and main results}

Let $f(x)=0$ be a real equation  and let $s\in\R$ be one of its solutions, i.\,e. such that $f(s)=0.$ It is said that  $h$ provides a {\it suitable iterative method for finding $s$} if for any $x_0,$ close enough to $s,$ it holds that the sequence
$x_{m+1}=h(x_m),$  $m \ge0,$ is such that $\lim_{m\to\infty}x_m=s.$ 
Moreover, it is said that this method $h$ \emph{ has order $p>1,$ with $p\in\R,$ towards $s$} if
\[
\lim_{m\to\infty}\frac{|x_{m+1}-s|}{|x_m-s|^p}=K\ne0.
\]
If in the above limit $K\ge 0$ then people say that  $h$ gives a method {\it   of order at least $p$ towards~$s$.} When $h$ is of class $\mathcal{C}^n, n\ge2,$ at $s$ it is well known than when it holds that  
\[
h(s)=s,\, h'(s)=h''(s)=\cdots=h^{(n-1)}(s)=0\quad\mbox{and}\quad h^{(n)}(s)\ne0,
\]
the method associated to this $h$ has order $p=n$ towards $s,$ see for instance \cite{SB}. This type of methods are called {\it without memory} because each $x_{m+1}$ depends only on the previous $x_m.$ There are methods with memory where $x_{m+1}$ is computed from some values $x_n$ with $n\le m.$ A paradigmatic method of this type is the celebrated Secant method, where $x_{m+1}$ is computed by using $x_m$ and $x_{m-1}.$ In this paper we only will consider methods without memory.

It is clear that a method of order $p_1,$ with $p_1>p_2>0,$  uses less steps for finding $s$ with a given accuracy that another one of order $p_2.$ Moreover, among several methods of a given order, the  one that needs less number of steps is the one that has associated the smaller $K,$ because for $n$ big enough $|x_{n+1}-s|\approx K|x_n-s|^p.$    
Nevertheless, instead of focusing  on the corresponding constants $K,$ usually, when comparing methods, people introduce the so called {\it efficiency,}  $E.$  The higher the efficiency, the better the method. 

This efficiency $E,$ takes into
account not only the order of convergence $p$ of the method  but also the time
$t$ needed for computing $x_{m+1}$ from $x_m.$  There are several ways of introducing it. For more details see for instance the papers~\cite{KT,KT2,tra0,tra2,var0}, the books~\cite{R} and  \cite[p. 12] {tra}, or  Section~\ref{se:eff}. A first way is called {\it informational efficiency} and it is introduced as $E=p/t$ and a second way
it is named {\it efficiency index} and it is $E=p^{1/t}.$ Notice that the value of $K$  is not taken into account in any of these definitions.

We will use the second one and we simply will call {\it efficiency}. We explain the main motivations for our choice: the function $p^{1/t}$ is simple, increasing  with respect to $p$ and decreasing with respect to $t$ and, more importantly, it satisfies  $E(p,t)=E(p^k, tk)$ for all $k\in\N.$  Let us interpret this equality. Consider a new method $H=h^{\circ k},$ where   $h^{\circ m}(x)=h(h^{\circ (m-1)}(x)).$  If $h$ has order $p$ and the time consumed in each iteration is $t$, it is clear that $H$ has order $p^k$ and  the time consumed in each iteration is $kt.$ Both methods should have the same efficiency because $H$ simply provides a subsequence of $h$ and this is precisely what the equality shows.

Hence, when comparing two methods of the same order the only factor needed to compare them is  the time $t$ used to perform an iteration, or in a few words the ``complexity'' of the map $h.$  Moreover we will not consider each equation individually. Rather we will take into account all  equations $f(x)=0$ having a simple zero $s$ and consider only methods $h_f$ of a fixed integer order and satisfying certain constrains. As we will 
discuss in Section~\ref{se:eff} there are different ways to estimate $t$ when studying the efficiency of the methods.

In this paper, fixed any \emph{integer} order  $p=n\ge 2,$ we will consider two different classes of methods. In a few words they are:
\begin{itemize}
	\item  {\bf Class I:\,} Methods that depend on $x,f(x)$ and its successive derivatives  until $f^{(n-1)}(x).$ 
	\item   {\bf Class II:\,}  Methods that depend on $x,g(x)$  until  $g^{\circ n}(x),$ where $g(x)=f(x)+x.$ 
\end{itemize}

Let us comment on an underlying common idea to construct examples of any order in both classes of methods. If $f(s)=0$ and $s$ is a simple zero then $s=f^{-1}(0).$ In Class I we approach $f^{-1}(y)$ by computing its Taylor polynomial of degree $n-1$ at a suitable point, while in Class II we approach $f^{-1}(y)$ by  its interpolation  polynomial of degree $n-1$ on an adequate sequence of points. Later we will develop these ideas   in more detail.

It is also important to remark that while methods of Class I only work for equations $f(x)=0,$ with $f$ having at least $n-1$ derivatives,  the methods of Class II can also be applied to equations given by continuous equations, $f(x)=0.$ 

 Inside the first class there are the celebrated Newton method ($n=2$) and the Chebyshev and Halley methods  ($n=3$), while inside the second one, there is the  Steffensen method ($n=2$). For completeness we show their well known  associated functions $h$:
 \begin{align*}
 	&h_{\mbox{\scriptsize New}}(x)=x-\frac{f(x)}{f'(x)}, && h_{\mbox{\scriptsize Che}}(x)=x-\frac{f(x)}{f'(x)}  -\frac{(f(x))^2f''(x)}{2 (f'(x))^3}, \\ &h_{\mbox{\scriptsize Hal}}(x)=x- \frac{2f(x)f'(x)}{2 (f'(x))^2 -f(x)f''(x)}, &&
 	h_{\mbox{\scriptsize Ste}}(x)=x-\frac{(g(x)-x)^2}{g(g(x))-2g(x)+x}.
 \end{align*}

This is our first main result.

\begin{teoa}\label{thm:a}  Let $s$ be a simple solution of an equation $f(x)=0,$ with $f:\R\to\R$  of class~$\mathcal{C}^{n}$. Consider the so called  canonical method $x\to H_f^n(x)$ where
\begin{align}\label{class1}
	H_f^n(x)&= x+\sum_{i=1}^{n-1}\frac{(-1)^i}{i!}(f(x))^i(f^{-1})^{(i)}(f(x))\nonumber\\&=:H^n(x,f(x),f'(x), f''(x), \ldots, f^{(n-1)}(x)).
\end{align}
The following holds:

\begin{enumerate}
	\item[(i)] The function $H^n:\R^{n+1}\to\R$ is a rational function with rational coefficients.

	\item [(ii)] The rational function $H^n$ is a polynomial function $P^n$ with integer coefficients in the new $n$ variables 
	\begin{equation}\label{eq:var}
		y_1=x,\,\, y_2=\frac{f(x)}{f'(x)}\mbox{ and, when } n\ge3,\,  y_j=\frac{f^{(j-1)}(x)}{(j-1)!f'(x)}\, \mbox{ for }\, j=3,4,\ldots, n,
	\end{equation}
	that is
	\[
	H_f^n(x)= H^n(x,f(x),f'(x), f''(x), \ldots, f^{(n-1)}(x))= P^n(y_1,y_2,\ldots,y_{n}),
	\]
	for some $P^n\in \Z[y_1,y_2,\ldots,y_{n}]$  of degree $2n - 3$, with degree $n - 1 $ in the variable $y_2$.

		\item [(iii)] If $f$  is of class $\mathcal{C}^{2n-2}$ the canonical method has order at least $n$ towards $s$ and moreover, for most functions $f$ (the ones satisfying $(f^{-1})^{(n)}(0)\ne0$) its order is exactly $n.$

	\item [(iv)] Any method of order at least $n$ towards $s$ writes as
	\begin{multline*}
	H^n(x,f(x),f'(x), f''(x), \ldots, f^{(n-1)}(x))\\+ \big(f(x))^n A(x,f(x),f'(x), f''(x), \ldots, f^{(n-1)}(x))
	\end{multline*}
	where $A:\R^{n+1}\to \R$ is a continuous function  such that $x\to x_2^nA(x_1,x_2,\ldots,x_{n+1})$ is of class~$\mathcal{C}^n.$ 
\end{enumerate}	
	\end{teoa}	

We remark that in item (iv) when we consider a general method of Class~I we refer to methods of the form $x\to H(x,f(x),f'(x), f''(x), \ldots, f^{(n-1)}(x))$ as is defined in detail in next section.

The above result is classical, see for instance~\cite{KT, ref1,tra}.  Nevertheless, our proof is simple, self-contained and based on a more formal and detailed  definition of Class~I methods. Moreover the regularity of the involved functions is sharpened.

As we have already said, the canonical methods for $n=2,3$ are the celebrated Newton and Chebyshev methods. Notice their compact expressions in terms of their corresponding polynomials, as  described in item~(ii):
\[
P_{\mbox{\scriptsize New}}(y)=y_1-y_2,\quad P_{\mbox{\scriptsize Che}}(y)=y_1-y_2-y_3y_2^2.
\]
For a similar compact expression of the canonical methods for $n\le7$ or for Halley method, see Section~\ref{se:nc} or~\cite{tra}. For completeness we  comment that Chebyshev method is also attributed  to Euler, see~\cite{OG,tra0} and their references for a full discussion about the matter.

In Theorem~\ref{t:equival}, we will give an improvement of statement (iii) of Theorem~\ref{thm:a}. More concretely, we prove that if  $f$ in the above theorem is merely of class~$ \mathcal{C}^n$, the iterative methods of Class I still converge to $s$ with order at least $n$.

Our second result gives a natural family of methods of order $n$ in Class~II that extend Steffensen method. To state it we need to introduce some more notations. For convenience and because of the usual notation of increments  we define $\Delta(x)=g(x)-x.$ Notice that in fact $f=\Delta$ but in many places will be more convenient to use this new name. Similarly, for any $x_0\in\R$ we define $x_m= g^{\circ m}(x_0) $ and associated  to this sequence we introduce for $m\ge0,$
\begin{equation}\label{eq:set}
\Delta_m=\Delta(x_m)= g(x_m)-x_m= x_{m+1}-x_m.
\end{equation}

One of the ideas of the method is that if $s$ is a solution of $f(x)=0,$ then $s$ is a fixed point of $g(x).$  Then some points of the discrete dynamical system generated by $g$ are used to construct the iterative method. 

The set of equations~\eqref{eq:set} can be inverted giving that 
\[
x_m= x_0+ \Delta_0+\Delta_1+\cdots+\Delta_{m-1}.
\]
With these notations in mind, and taking $x_0$ instead of $x,$ the above definition of Class~II methods can be formulated in an equivalent but more suitable way:

\begin{itemize} 
	\item   {\bf Class~II:\,}  Methods that depend on $x_0,\Delta_0, \Delta_1, \Delta_2,$  until  $\Delta_{n-1}.$ 
\end{itemize}

Finally we compute the interpolation polynomial  $P_{n-1}(y)$ of degree $n-1$  of $f^{-1}(y)$ at the $n$ points of Table~\ref{tab1}. Notice that
\[
f(x_0+\Delta_0+\Delta_1+\cdots+\Delta_{m-1})= f(x_m)=g(x_m)-x_m=x_{m+1}-x_m=\Delta_m.
\]

\begin{table}[h]
	\begin{center}
		\begin{tabular}{|c|| c| c|c|c|c|c} 
			\hline
			$y$ & $\Delta_0$&  $\Delta_1$&  $\Delta_2$& $\cdots$ & $\Delta_{n-1}, \, n\ge 2$ \\
			\hline
			$f^{-1}(y)$ & $x_0$& $x_0+\Delta_0$&$x_0+\Delta_0+\Delta_1$& $\cdots$ & $x_0+\Delta_0+\Delta_1+\cdots+\Delta_{n-2}$\\
			\hline
		\end{tabular}
		\caption{The function $f^{-1}$ at $n$ points.}\label{tab1}
	\end{center}
\end{table}

By using the divided differences Newton method   we get that
\[
P_{n-1}(y)=x_0+\sum_{i=1}^{n-1} D_{0,1,\ldots,i} (y-\Delta_0)(y-\Delta_1)\cdots(y-\Delta_{i-1}),
\]
where $D_{0,1,\ldots,m}$ are the usual divided differences, defined recurrently as
\[
D_{0,1,\ldots,m}=\frac{D_{1,\ldots,m}-D_{0,1,\ldots,m-1}}{\Delta_m-\Delta_0},
\]  
with  $D_0=x_0$ and $D_m=x_0+\Delta_0+\Delta_1+\cdots+\Delta_{m-1},$ for $m\ge1.$

Then, it can be seen that
\begin{equation}\label{ste-gen}
P_{n-1}(0)=: K_f^n(x_0)= x_0+\sum_{i=1}^{n-1} \dfrac{Q_{i}(\Delta_0,\Delta_1,\ldots,\Delta_{i})}{R_{i}(\Delta_0,\Delta_1,\ldots,\Delta_{i})}\Delta_0\Delta_1\dots\Delta_{i-1},
\end{equation}
where, for each $m,$ $Q_m$ is a homogeneous polynomial and $R_m$ is the homogeneous polynomial formed by the product of the $m(m+1)/2$ factors $\Delta_j-\Delta_i$ with $0\le i<j\le m.$ In Section~\ref{ste-gen} after the proof on next theorem, we give the expressions of $K_f^n(x_0)$ for low values of $n.$

\begin{teoa}\label{thm:b}  For $n \geq 2$ and any $\mathcal{C}^n$ function $f$ with a simple zero $s$, the method of Class~II, given by the map $x_0 \mapsto K_f^n(x_0)$, where $K_f^n$ is defined in~\eqref{ste-gen} and all expressions involved in this map have been introduced above, is a rational method that converges to $s$ with order $n$.
\end{teoa}

\section{Proof of Theorem~\ref{thm:a}}

We start this section with the proof of a preliminary result that affirms that any sufficiently smooth  function defined in a convenient open subset of $\R^n$ can be expanded in powers in one coordinate function. We have chosen to write this result in terms of the first coordinate but clearly it is also valid for any election of coordinate.

\begin{lem} \label {taylor} Let $\mathcal{U}$ be an open subset of $R^{n+1},\,\,n\ge 1,$ satisfying that $(sx_0,x_1,\ldots,x_n)\in \mathcal{U}$ when $(x_0,x_1,\ldots,x_n)\in \mathcal{U}$ and $s\in[0,1].$ Set $g:\mathcal{U}\longrightarrow\R$ be a $\mathcal C^k,\,\,k\ge 1$ function. Then for any $0<m\le k$ we have 
	\begin{align*}
		g(x_0,x_1,\ldots,x_n)&=g_0(x_1,\ldots,x_n)+x_0g_1(x_1,\ldots,x_n)+\ldots+ x_0^{m-1}g_{m-1}(x_1,\ldots,x_n)\\ &\quad +x_0^mh_m(x_0,x_1,\ldots,x_n)
	\end{align*}
where for $i<m,\,\, g_i(x_1,\ldots,x_n)=\frac{\partial^i g}{i!\partial x_0^i}(0,x_1,\ldots,x_n)$ is defined in an open set of $R^n$ and is of class $\mathcal C^{k-i}$ and $h_m$ is defined in $\mathcal{U}$ and is of class $\mathcal C^{k-m}.$
\end{lem}

\begin{proof}
 We prove the lemma by finite induction over $m.$ For $m=1$ we have 
\begin{align*}  g(x_0,x_1,\ldots,x_n)-g(0,x_1,\ldots,x_n)&=\int_0^1\frac {{\rm d}}{{\rm d}s}g(sx_0,x_1,\ldots,x_n)\,{\rm d}s\\&=x_0\int_0^1D_1g(sx_0,x_1,\ldots, x_n)\,{\rm d}s
\end{align*}	
	and the lemma follows in this case putting $g_0=g(0,x_1\ldots,x_n )$ which is $\mathcal C^k$ and $h_1=\int_0^1D_1g(sx_0,x_1,\ldots, x_n)\,{\rm d}s$ which clearly is $\mathcal C^{k-1}.$

Now assume that the result is proved for some $m<k$ and we prove it for $m+1.$ We get 	\begin{align*}
	g(x_0,x_1,\ldots,x_n)&=g_0(x_1,\ldots,x_n)+x_0g_1(x_1,\ldots,x_n)+\ldots+ x_0^{m-1}g_{m-1}(x_1,\ldots,x_n)\\ &\quad +x_0^mh_m(x_0,x_1,\ldots,x_n)
\end{align*} with $h_m$ of class $\mathcal C^{k-m}$ where $k-m\ge 1.$ Applying the Lemma to $h_m,$ we have
$$h_m=f_0(x_1,\ldots,x_n)+x_0f_1(x_0,x_1,\ldots,x_n)$$ where $f_0$ is of class $\mathcal C^{k-m}$ and $f_1$ is of class $\mathcal C^{k-m-1}.$ Then the lemma follows for $m+1$ putting $g_m=f_0$ and $h_{m+1}=f_1.$ Note that from these expansions it follows that for $i=1,\ldots,k-1,\,\,i!g_i(x_1,\ldots,x_n)=\frac{\partial^i g}{\partial x_0^i}(0,x_1,\ldots,x_n).$
\end{proof}

Before starting our proof of Theorem~\ref{thm:a}, let us  give the precise definition of Class I methods.
\bigskip

\noindent {\bf Class I methods:\,}
For $n\ge 2,$ set $\Pi_n=\R^{n+1}\setminus \{(x_0,x_1,\ldots,x_{n})\in \R^{n+1} \mbox{ such that } x_2\ne 0\}$ and $\Gamma_n= \{(x_0,x_1,\ldots,x_{n})\in\Pi_n\mbox{ such that } x_1=0\}.$ A {\em Class I  method of order~$n$ to approximate simple zeroes of a real function} $f$ is a pair $(H,\mathcal{W}_H)$ where $\mathcal{W}_H$ is an open neighborhood of $\Gamma_n$ in $\Pi_n$ and $H$ is a ${\mathcal C}^{n}$ function $H:\mathcal{W}_n\longrightarrow \R$ satisfying
the following properties. \begin{enumerate} \item[(a)] For all $(x_0,0,\ldots,x_{n})\in\Gamma_n,\,\,H(x_0,0,x_2,\ldots,x_{n})=x_0.$ 
	\item[(b)] For any real ${\mathcal C}^{2n-2}$  function $f,$ the map $H_f(x)=H(x,f(x),f'(x),\ldots,f^{(n-1)}(x)),$ which is well defined in a neighborhood of any simple zero $s$ of $f,$ satisfies  $H_f^{(i)}(s)=0$  for any $i\in \{1,\ldots,n-1\}.$
		\item[(c)] If $(x_0,x_1,\ldots,x_{n})\in \mathcal{W}_n$ then $(x_0,tx_1,\ldots,x_{n})\in \mathcal{W}_n$ for all $t\in[0,1]$.
\end{enumerate}

 We remark that the name  {\it Class I method of order~$n$} is adequate. This so  because, as we will prove in Theorem~\ref{t:equival}, when the introduced functions $H$ are transformed into the corresponding maps $H_f$, and the functions $f$ are of class at least $\mathcal{C}^{n},$ they provide {\it proper methods of order $n.$}

A last preliminary result is:

\begin{propo}\label{unici} Let $G,H$ two methods of Class I of order $n.$  Then
	$$(G-H)(x_0,x_1,\ldots,x_{n})=x_1^{n}F(x_0,x_1,\ldots,x_{n})$$ where $F:\mathcal{W}_G\cap \mathcal{W}_H\longrightarrow \R$ is of class $\mathcal C^0$ and $ x_1^nF$ is of class $\mathcal C^n.$ \end{propo}

\begin{proof} From Lemma \ref{taylor} we have
	$$(G-H)(x_0,x_1,\ldots,x_{n})=\sum_{i=0}^{n-1}x_1^if_i(x_0,x_2,\ldots,x_{n})+x_1^{n}F(x_0,x_1,\ldots,x_{n})$$ with $f_i$ of class $\mathcal C^{n-i}$ and $F$ of class $\mathcal C^0.$ Suppose, to arrive a contradiction,
	that some of the $f_i$'s are not identically zero and let $j$ be the first integer such that $f_j$ is not identically zero. Set $(s_0,s_2,\ldots,s_{n})$ satisfying $f_j(s_0,s_2,\ldots,s_{n})\ne 0$ and let $\varphi$ be a $\mathcal C^{\infty}$ function having a simple zero at $s_0$ and satisfying that $\varphi^{(i)}(s_0)=s_{i+1}$ for $i=2,\ldots, n-1.$ 
	Simple computations show that $$(G-H)_{\varphi}^{(j)}(s_0)=j!\varphi ' (s_0)f_j(s_0,s_2,\ldots,s_{n})\ne 0$$ which contradicts the fact that $G$ and $H$ are methods of order $n.$ So all the functions $f_i$ are identically zero and the result follows.
\end{proof}

\begin{proof}[Proof of Theorem~\ref{thm:a}]  
(i) Recall that 	
\begin{align*}
	H_f^n(x)&= x+\sum_{i=1}^{n-1}\frac{(-1)^i}{i!}(f(x))^i(f^{-1})^{(i)}(f(x))=H^n(x,f(x),f'(x), f''(x), \ldots, f^{(n-1)}(x)).
\end{align*}
	We will prove by induction that $H^n:\Pi_n\longrightarrow \R$ is a rational function. 	More precisely, we claim that 
\begin{equation}\label{e:claim}
(f^{-1})^{(i)}(f(x))=\frac{U_{i-1}(f'(x),f''(x),\ldots,f^{(i-1)}(x))}{f'(x)^{2i-1}},
\end{equation}
where $U_{i-1}$ is a homogeneous polynomial of degree $i-1.$ Since $(f^{-1})'(f(x))=\frac{1}{f'(x)}$, the claim is true for $i=1$. Now assume that the claim holds for $i$ and we will prove it for $i+1.$ Since $$(f^{-1})^{(i+1)}(f(x))=\frac{((f^{-1})^{(i)}(f(x)))'}{f'(x)}$$ we get 
	
	\begin{align*} (f^{-1})^{(i+1)}(f(x))& =\frac{1}{f'(x)}\left(U_{i-1}(f'(x),f''(x),\ldots,f^{(i-1)}(x))f'(x)^{1-2i}\right)'\\ &=\frac{1}{f'(x)}\left(  U_{i-1}(f'(x),f''(x),\ldots,f^{(i-1)}(x))\right)'f'(x)^{1-2i}+\\ &\quad \frac{1}{f'(x)}\left( (1-2i) U_{i-1}(f'(x),f''(x),\ldots,f^{(i-1)}(x))f'(x)^{-2i}f''(x)\right) \end{align*}
	
	Since, by the chain rule, $\left(U_{i-1}(f'(x),f''(x),\ldots,f^{(i-1)}(x))\right)'$ is still a homogeneous polynomial in the variables $f'(x),f''(x),\ldots,f^{(i)}(x)$ of degree $i-1,$ the claim holds for $i+1$ putting $$U_i= \left(U_{i-1}(f'(x),f''(x),\ldots,f^{(i-1)}(x))\right)'f'(x)+(1-2i) U_{i-1}(f'(x),f''(x),\ldots,f^{(i-1)}(x))f''(x).$$ This ends the proof of the claim and, as a consequence, item (i) holds.
	
\bigskip	
	
(ii)	That $H^n$ is a polynomial in the variables $y_1,\ldots,y_n$ follows directly from the above claim (see Equation \eqref{e:claim}) and the expression of $H^n.$ To show that the coefficients are integers, using the previous notation, we write $$U_{i-1}(f'(x),f''(x),\ldots,f^{(i-1)}(x))=\sum_{j_1+\ldots +j_{i-1}=i-1}a^i_{j_1\ldots j_{i-1}}(f'(x))^{j_1}\cdots(f^{(i-1)}(x))^{j_{i-1}}.$$ Now it suffices to show that for all $i$ we have 
\begin{equation}\label{enters} a^i_{j_1\ldots j_{i-1}}(2!)^{j_2}\cdots((i-1)!)^{j_{i-1}}\in \Z.
\end{equation}

We prove (\ref {enters}) by induction.
For $i=3$ we have 
\begin{align*}H^3(x,f(x),f'(x),f''(x),f^{(3)}(x))&=x-\frac{f(x)}{f'(x)}-\frac{f''(x)}{2(f'(x))^3}(f(x))^2\\ &\quad-\frac{3(f''(x))^2-f^{(3)}(x)f'(x)}{3!(f'(x)^5)}(f(x))^3\end{align*} 	
so with the above notation we get $U_0=1, U_1=-f''(x)/2$ and $a^1_{01}=-1/2,\,\, U_2=f''(x)^2/2!-f^{(3)}(x)f'(x)/3!$ and    $ a^2_{020}=1/2, a^2_{101}=-1/3!.$ Therefore for $i=2,3$  relation (\ref{enters}) holds.	

Now assume that (5) is true for $i\ge 3.$ To show that the relation is also true for $i+1$ we will see that all the monomials appearing in the formal procedure to obtain  $U_{i}$ still satisfies the relation. Note that a monomial, $a^i_{j_1\ldots j_{i-1}}(f'(x))^{j_1}\cdots(f^{(i-1)}(x))^{j_{i-1}},$ appearing in $U_{i-1}$ contributes in $U_i$ with a generic term $$j_ka^i_{j_1\ldots j_{i-1}}(f'(x))^{j_1+1}\cdots(f^{(k)}(x))^{j_k-1}(f^{(k+1)}(x))^{j_{k+1}+1}\cdots(f^{(i-1)}(x))^{j_{i-1}}$$ and also with
$$(1-2i)a^i_{j_1\ldots j_{i-1}}(f'(x))^{j_1}(f''(x))^{j_2+1}\cdots(f^{(i-1)}(x))^{j_{i-1}}.$$

Since by the induction hypothesis  $a^i_{j_1\ldots j_{i-1}}(2!)^{j_2}\cdots((i-1)!)^{j_{i-1}}\in \Z$ we get
\begin{multline*}
	j_ka^i_{j_1\ldots j_{i-1}}(2!)^{j_2}\cdots(k!)^{j_k-1}((k+1)!)^{j_{k+1}+1}\cdots((i-1)!)^{j_{i-1}}\\
	=j_k\frac{(k+1)!}{k!}a^i_{j_1\ldots j_{i-1}}(2!)^{j_2}\cdots(k!)^{j_k}((k+1)!)^{j_{k+1}}\cdots((i-1)!)^{j_{i-1}}\in \Z.
\end{multline*}
and 
$$(1-2i)a^i_{j_1\ldots j_{i-1}}(2!)^{j_2+1}\cdots((i-1)!)^{j_{i-1}}\\
=(1-2i)2a^i_{j_1\ldots j_{i-1}}(2!)^{j_2}\cdots((i-1)!)^{j_{i-1}}\in \Z.$$ This ends the induction to prove relation (\ref {enters}) and the proof of item~(ii).
	
\bigskip
	
(iii) Recall that in this item we assume that $f$ is a class $\mathcal C^{2n-2}$ function. Notice that
\[
\left(\frac{(-1)^i}{i!}(f(x))^i(f^{-1})^{(i)}(f(x))\right)'=W_i(x)-W_{i-1}(x),
\]
where
\[		W_i(x)=\frac{(-1)^i}{i!}(f(x))^i(f^{-1})^{(i+1)}(f(x))f'(x)\quad \mbox{with} \quad W_0(x)=1.
\]
Hence by the telescopic property  of the sum we get that
\[
\big(H_f^n(x)\big)'=1+W_{n-1}(x)-W_0(x)=W_{n-1}(x)=  \frac{(-1)^{n-1}}{(n-1)!}(f(x))^{n-1}(f^{-1})^{(n)}(f(x))f'(x).
\]
Since $f(s)=0,$ we obtain that $\big(H_j^n\big)^{(i)}(s)=0,$ for $i=1,2,\ldots,n-1$ as we wanted to prove. Moreover,
$
\big(H_f^n\big)^{(n)}(s)=(-1)^{n-1}(f^{-1})^{(n)}(0)(f'(s))^n $ is not zero for most functions. Hence item (iii) follows.

(iv) It follows directly from Proposition \ref{unici}.
\end{proof}

\begin{teo}\label{t:equival} 
	Let $H_f(x)=H(x,f(x),f'(x),\ldots,f^{(n-1)}(x))$ be a method of Class I of order~$n,$ and let $f$ be a class $\mathcal C^n$ function  with a simple zero at $s.$ 	Then the iterative method $H_f$ converges to $s$ with order at least $n.$ 
\end{teo}

\begin{proof} 
From item (iv) of Theorem~\ref{thm:a} we have that $$ H_f(x)=H^n_f(x)+(f(x))^nA(x,f(x),f'(x),\ldots,f^{(n-1)}(x))$$ for a certain continuous map $A:\Pi_n\longrightarrow \R$ satisfying that $x_1^nA$ is of class $\mathcal C^n.$ Notice that the points in $\Pi_n$ write as $(x_0,x_1,\ldots,x_n)\in\R^{n+1}.$ From now on we will write $A_f(x)=A(x,f(x),f'(x),\ldots,f^{(n-1)}(x)).$  Note that since $H$ is a method of order $n\ge 2$ we have that $s$ is a fixed point of $H_f$ with $H_f'(x)=0.$ So it is locally asymptotically stable and then for $x$ close enough to $s$ the sequence $x_m=H_f^{\circ m}(x)$ converges to $s.$ 
 Define $y=f(x).$ Since $f(s)=0$ and $f'(0)\ne0$ near $y=0,$ then  $x=f^{-1}(y)$ is well-defined and, moreover, $s=f^{-1}(0).$ If $x$ is  near enough to $s$ we know that $f(x)$ it is also close enough to $y=0.$ By using Taylor formula for $f^{-1}$ at $y=f(x)$ we have that
	$$
	f^{-1}(y)= \sum_{i=0}^{n-1}\frac{(f^{-1})^{(i)}(f(x))}{i!} (y-f(x))^i+\frac{(f^{-1})^{(n)}(\eta)}{n!} (y-f(x))^n,
	$$
	where  $\eta=\eta_{y,x}$ is between $y$ and $f(x).$	 By taking $y=0$ above we get
	\begin{align*}
	s &= H_f^n(x)+\frac{(f^{-1})^{(n)}(\mu_x)}{n!} (-f(x))^n\\&=H_f(x)-f(x)^nA_f(x)+\frac{(f^{-1})^{(n)}(\mu_x)}{n!} (-f(x))^n\\
	&=H_f(x)-\Big((-1)^nA_f(x)-\frac{1}{n!}(f^{-1})^{(n)}(\mu_x)\Big)(f(s)-f(x))^n\\&= H_f(x)-\Big((-1)^nA_f(x)-\frac{1}{n!}(f^{-1})^{(n)}(\mu_x)\Big)\big( f'(\nu_x)(s-x)\big)^n
\end{align*}
	where $\mu_x=\eta_{0,x}$ is between $0$ and $f(x),$ and $\nu_x$ is between $s$ and $x,$ and we have also applied the mean value theorem.
	
	Since $x_{m+1}=H^n_f(x_m),$ the above equality gives that
	\begin{align*}
	\lim_{m\to\infty}\frac{|s-x_{m+1}|}{|s-x_m|^n} &= \lim_{m\to\infty}\Big|\Big((-1)^nA_f(x_m)-\frac{1}{n!}(f^{-1})^{(n)}(\mu_{x_m})\Big)(f'(\nu_{x_m}))^n\Big|\\ &=\Big|\Big((-1)^nA_f(s)-\frac{1}{n!}(f^{-1})^{(n)}(0)\Big)(f'(s))^n\Big|
	\end{align*}
   because $A$ is continuous, $f$ is of class $\mathcal C^n$ and $\lim_{m\to\infty}x_m=s$ if $x$ is close enough to $s.$ If this limit is different from zero then $H_f$ is of order $n.$ Otherwise $H_f$ has order at least $n.$	
\end{proof}

\section{Efficiency and some explicit methods in Class~I}\label{se:eff}

In this section we start with some comments about efficiency, including in particular, the celebrated  Kung-Traub's conjecture. We also study the efficiency of some simple methods in the light of this concept and the results of Theorem~\ref{thm:a}.

\subsection{Around the definition of efficiency and Kung-Traub's conjecture}

To understand  the origin of the definition of efficiency $E(p,t)=p^{1/t}$ (or other equivalent ones given below) we reproduce the argument given in~\cite[Sec. 3]{G}. As we have argued in the introduction, $E$ must satisfy the identity  $E(p,t)= E(p^k, tk)$ for all $k\in\N.$ Imposing the same equation, but for all $x\in\R^+$ we get  $E(p,t)= E(p^x, tx).$ Then, making the derivative with respect to $x$ we obtain
\[ p^x
\ln(p)\frac{\partial E}{\partial p}\left(p^x, tx\right)+ t \frac{\partial E}{\partial t}\left(p^x, tx\right)=0.
\]
By replacing $x=1$, we arrive to the linear partial differential  equation
\[
p\ln(p) \frac{\partial E}{\partial p}(p, t)+ t \frac{\partial E}{\partial t}(p, t)=0,
\]
for which all explicit solutions are $E(p,t)=\varphi( \ln(p)/t),$ where $\varphi$ is any smooth function. By taking $\varphi(y)=y, \varphi(y)= y/\ln(2)$ or $\varphi(y)=\exp(y)$ we arrive to three natural, and essentially equivalent, definitions of efficiency, $\ln(p)/t$, $\log_2(p)/t$ or $p^{1/t}.$ The third one is the one chosen in this paper although several people also utilize the second one. Anyway all them are equivalent.  In fact, in \cite{KT2,tra2} it is also commented without proof that $\log_2(p)/t$  is essentially the unique natural definition of efficiency  and this assertion is attributed to Gentleman in a 1970 private communication to the authors.

In any case, and among methods of the same order $p,$ the key point to know the efficiency is how to estimate $t.$ We collect some comments about this question following \cite{GNS,KT,KT2, tra,tra2,var0} and their references.  Following \cite{KT} to compare various algorithms, it is necessary to define efficiency measures based on the speed of
convergence ({\it order $p$}), the cost of evaluating of the involved function and its derivatives ({\it problem cost}), and the cost
of forming the iteration making elementary operations: $+,-,*,/$ ({\it combinatorial cost}). Clearly the combination of the last two  factors leads to $t.$ 

One  of the followed ways is to concentrate in $t$ only the problem cost, simply counting the number of evaluations of $f$ and its derivatives at different points, say $d$. Hence, from this point of view, Newton method would have $d=2$ because at each iteration uses $f(x)$ and $f'(x),$ and then $t=d=2$ and its efficiency would be $2^{1/2}.$ Similarly Chebyshev method would have efficiency $3^{1/3}.$  A more efficient method with this criterion would be Ostrowski method (\cite{O}). It is
\[
x_{n+1}=y_n-\frac{f(y_n)(x_n-y_n)}{f(x_n)-2f(y_n)},\quad \mbox{where}\quad y_n=x_n-\frac{f(x_n)}{f'(x_n)}.
\]
Notice that $d=3$ because it uses $f(x_n), f'(x_n)$ and $f(y_n)$ and it can be easily seen that has order $p=4.$ Hence its efficiency is higher and it is $4^{1/3}.$ In general, methods with~$d$ evaluations and order $2^{d-1},$ for any $d\ge2$ are known, their corresponding efficiency is $2^{(d-1)/d}$ and are called optimal methods, see for instance \cite{var} and its references. The methods where the function~$f$ and its derivatives are evaluated at different points are called {\it multipoint methods.}

In \cite{KT,KT2} it  is stated the nowadays known as Kung--Traub's Conjecture: {\it A multipoint iteration without memory based on $d$ evaluations has optimal order $2^{d-1}.$} 

There are also many papers where in the computation of $t$ the two elements,  problem and combinatorial costs are taken into account, see for instance~\cite{GNS, GNS2,KT2} and their references.

In next section we will study the efficiency of some methods of order $p=n$ inside Class~I. Notice that Theorem~\ref{thm:a} ensures that this is the higher order in this class. Recall that all them use $f,f',f''$ until $f^{(n-1)}$ at the same point, that is, for all them $d=n.$   Hence it has sense to discard this part in the computation of $t,$ because all them use the same values and then $t$ is merely reduced to the value introduced above as combinatorial cost, that only takes into account the number of elementary operations. In all cases, we will use the variables introduced in \eqref{eq:var}.

\subsection{Canonical methods of order at most $7.$}\label{se:nc} 

By using item (ii) of Theorem~\ref{thm:a}, the canonical methods are
\[
H_f^n(x)= H^n(x,f(x),f'(x), f''(x), \ldots, f^{(n-1)}(x))= P^n(y_1,y_2,\ldots,y_{n})=P^n(y),
\]
where $H_f$ is introduced in~\eqref{class1}.
Long but straightforward computations give that the general expression using the variables \eqref{eq:var} is:
\begin{align*}
	P^n(y) &= y_{{1}}-y_{{2}}-y_{{3}}{y_{{2}}}^{2}-\left( {2y_{{3}}}^{2}-y_{{4}} \right){y_{{2}}}^{3}  - \left( 5{y_{{3}}}^{3}-5y_{{3}}y_{{4}}+y_{{5
	}} \right){y_{{2}}}^{4}
	\\ &\quad  -\left( 14{y_{{3}}}^{4}-21{y_{{3}}}^{2}y_{
		{4}}+6y_{{3}}y_{{5}}+3{y_{{4}}}^{2}-y_{{6}} \right)y_{{2}}^{5} \\&\quad  -\left( 42{y_{{3}}}^{5}-{84}{y_{{3}}}^{3}y_
	{{4}}+28{y_{{3}}}^{2}y_{{5}}+28y_{{3}}{y_{{4}}}^{2}-7y_{{3}}y
	_{{6}}-7y_{{4}}y_{{5}}+y_{{7}} \right) y_{{2}}^{6} \\&\quad+\cdots+ Q^n(y_3,y_4,\ldots,y_{n})y_{{2}}^{n-1},
\end{align*}
for some polynomial $Q^n$ of degree $n-2.$ From the above expression we get explicitly all the canonical methods with order $n\le7.$

\subsection{Some rational methods}
In this section we present some simple rational methods based of Theorem~\ref{thm:a}. More precisely, the formula for the method in the variables $y,$ say $R^n,$ will be a rational expression on these variables with the condition that $R^n-P^n$ has order $n$ in the variable $y_2.$ In order to study their efficiency we will take into account the number of operations that we need to compute the iterates, and more in particular, the number of products and quotients.

For one variable polynomial expressions, it is known that Horner's rule, in which a polynomial is written in nested form:
$$a_0+a_1x+a_2x^2+a_3x^3+\cdots +a_nx^n=a_0+x\Big(a_1+x\big(a_2+x(a_3+\cdots+x(a_{n-1}+xa_n)\cdots)\big)\Big),$$
is the approach for computing efficiently the polynomials because it involves less sums and products. We will also use it, with some privileged variable, $y_2$ to compute the numerators and denominators of our rational expressions.

\subsubsection{Examples of order $3$}
Considering the methods 
$$R^3(y)=y_1+\frac{y_2+Ay_2^2}{-1+ay_2},$$
with $A,a$ polynomials in $y_3,$ the condition $R^3(y)-P^3(y)=O(y_2^3)$ implies that $A+a=y_3.$ Hence, the simplest example is when $A=0$ and $a=y_3,$ which provides the celebrated Halley's method
\begin{equation}\label{M}
	R^3(y)=y_1+\frac{y_2}{-1+y_3y_2}.
\end{equation}

Once $y_1$ and $y_2$ are computed this method uses 1 multiplication, 1 division and 2 additions/subtractions.  Similarly the Chebyshev method uses 2 multiplications and 2 subtractions. In the light of the computations and explanations of \cite[Tables 1-2]{GNS} in most computers  the division time is longer than the multiplication one and so the Chebyshev method is more efficient that the Halley one.

\subsubsection{Examples of order $4$} 
Consider the family of methods
$$y_1+\frac{y_2+Ay_2^2+By_2^3}{-1+ay_2+by_2^2},$$
where $A,B,a,b$ are polynomials in $y_3,y_4.$
In order to minimize the computations and get simple methods, we will consider three different cases.
\begin{enumerate}
	\item [(i)] Consider first $A=B=0.$ Then
	$$R^4_1(y)=y_1+\frac{y_2}{-1+ay_2+by_2^2}.$$ The condition $R_1^4(y)-P^4(y)=O(y_2^4)$ gives
	$a=y_3$ and $b=y_3^2-y_4,$ i.e.,
	\begin{equation}\label{M1}
		R^4_1(y)=y_1+\frac{y_2}{-1+y_3y_2+(y_3^2-y_4)y_2^2}=y_1+\frac{y_2}{-1+y_2(y_3+y_2(y_3^2-y_4))}.
	\end{equation}
	In the last equality we write the denominator in Horner's form. This method is the method $\psi_{1,2}$ which appears in \cite[p. 90]{tra}.
	\item [(ii)] Consider $b=0$ with $a\ne 0.$ Then $$R^4_2(y)=y_1+\frac{y_2+Ay_2^2+By_2^3}{-1+ay_2}.$$ Now the condition $R_2^4(y)-P^4(y)=O(y_2^4)$ gives
	$A=y_3-a$ and $B=-ay_3+2y_3^2-y_4.$ Hence, to simplify $B$ it is convenient take $a=2y_3$ and then $B=-y_4.$ Then,
	\begin{equation}\label{M2}
		y_1+\frac{y_2-y_3y_2^2-y_4y_2^3}{-1+2y_2y_3}=y_1+\frac{y_2(1-y_2(y_3+y_4y_2))}{-1+2y_3y_2}.
	\end{equation}
	\item[(iii)] Finally consider $B=0$ with $A\ne 0.$ Then $$R^4_3(y)=y_1+\frac{y_2+Ay_2^2}{-1+ay_2+by_2^2}.$$ Now it is necessary that $A=-a+y_3$ and $b=-ay_3+2y_3^2-y_4.$ If we take  $a=2y_3$ then $A=-y_3$ and $b=-y_4$ and
	\begin{equation}\label{M3}
		R_3^4(y)=y_1+\frac{y_2-y_3y_2^2}{-1+2y_3y_2-y_4y_2^2}=y_1+\frac{y_2(1+y_3y_2)}{-1+y_2(2y_3-y_4y_2)}.
	\end{equation}
\end{enumerate}
To implement the methods $R^4_1,R^4_2$ and $R^4_3$ given in (\ref{M1}),(\ref{M2}) and (\ref{M3}), once $y_1,y_2,y_3$ and $y_4$ are already computed,  it is necessary to perform three more products and one quotient, four products and one quotient and four products and one quotient respectively, as long as we consider the formulas expressed in the Horner's form. Notice that we do not count the simple operation multiply by $2$ as a multiplication. 

By similar considerations that the ones done in previous section, the more efficient methods of order 4 among the ones considered in this paper are firstly the canonical one (only uses four multiplications written in Horner's form with respect to $y_2$), followed by the one defined by~$R^4_1.$

\subsubsection{Halley's type methods}
From the results of the above two subsections it is very natural to consider  for each order $n,$ a method
$$ Q^n(y)=y_1+ \frac{y_2}{-1+\alpha_1y_2+\alpha_2y_2^2+\cdots +\alpha_{n-2}y_2^{n-2}}$$
where $\alpha_1,\alpha_2,\ldots ,\alpha_{n-2}$ are polynomials in $y_3,y_4,\ldots ,y_n$ determined by the condition $Q^n(y)-P^n(y)= O(y_2^n)$. We call these methods, Halley's type methods.
It is easy to see that for each order $n$ the polynomials $\alpha_1,\alpha_2,\ldots ,\alpha_{n-2}$ are uniquely determined. Calling $R^n$ the polynomial $-1+\alpha_1y_2+\alpha_2y2^2+\cdots +\alpha_{n-2}y_2^{n-2}$ the first  ones are:
\begin{align*}
	R^3(y)&=-1+y_3y_2,\\
	R^4(y)&=-1+y_2(y_3+y_2(y_3^2-y_4)),\\
	R^5(y)&=-1+y_2 \Big(y_3+y_2\big(y_3^2-y_4+y_2(2y_3^3-3y_3y_4+y_5)\big)\Big),\\
	R^6(y)&=-1+y_2\Big(y_3+y_2\big(y_3^2-y_4+y_2(2y_3^3-3y_3y_4+y_5\\&\quad+y_2(5y_3^4-10y_3^2y_4+4y_3y_5+2y_4^2-y_6))\big)\Big),
\end{align*}
where, obviously the first ones coincide with Halley and  $R_1^4(y)$ given in~\eqref{M1}.

\section{Proof of Theorem~\ref{thm:b}}

As we will see,  the proof of Theorem~\ref{thm:b} is based on proving first that the order of the involved methods is at least $2.$ This is proved in next preliminary result.

\begin{lem}\label{l:ordre2}  When $f$ is a  class $\mathcal{C}^2$ function, the canonical method $K^n_f$ given in~\eqref{ste-gen} is a rational method of order at least $2$
\end{lem}

\begin{proof} Without loss of generality, we take $s=0$. Set $g\in\mathcal{C}^2(\mathcal{U})$, where $0\in\mathcal{U}$ is an open set, then   $g(x)=ax+o(x)$, assume that $a\neq 1$. Our objective is to prove that
$(K^n_f)'(0)=0$. The proof will have two steps. In the first one, we will prove that for the method given in Eq. \eqref{ste-gen}, $K_f^n(g(x))=\kappa(a)x+o(x)$. In the second step we will prove that $\kappa(a)=0$.

  A simple computation gives:
\begin{equation}\label{e:delta1}
\left.\Delta_{i}\right|_{g(x)}=(a-1)a^ix+o(x)\quad \mbox{and} \quad
\left.\left(\Delta_{j}-\Delta_{i}\right)\right|_{g(x)}=(a-1)(a^j-a^i)x+o(x),
\end{equation}

\emph{Step 1}. Using Equations \eqref{e:delta1}  it is straightforward to obtain that: 
\begin{equation}\label{e:2bis}
\left.\left(x+\frac{Q_1(\Delta_0)}{R_1(\Delta_0,\Delta_1)}\Delta_0\right)\right|_{g(x)}=
\left.\left(x+\frac{\Delta_0^2}{\Delta_1-\Delta_0}\right)\right|_{g(x)}=o(x^2).
\end{equation}

Now, we claim that
\begin{equation}\label{e:4}
\left.\left(\frac{Q_i(\Delta_0,\ldots,\Delta_{i-1})}{R_i(\Delta_0,\ldots,\Delta_i)}\right)\right|_{g(x)}=d_i(a)\,x^{1-i}+o(x^{1-i}).
\end{equation}
Indeed, notice that for all $i\in\{1,\ldots,n\}$, 
$
{Q_i(\Delta_0,\ldots,\Delta_{i-1})}/{R_i(\Delta_0,\ldots,\Delta_i)}=(-1)^iD_{0,1,\ldots, i},
$
where $D_{0,1,\ldots, i}$ is the corresponding divided difference appearing in the interpolating polynomial of $f^{-1}(y)$.  The claim follows from the fact that for any divided difference 
$$
D_{i_1,\ldots,i_k}|_{g(x)}=d_{i_1,\ldots,i_k}|_{g(x)} x^{1-k}+o(x^{1-k}).
$$
To prove it, we proceed by induction. A computation shows that 
$$
D_{0,1}|_{g(x)}=\left.\left(\frac{\Delta_0}{\Delta_1-\Delta_0}\right)\right|_{g(x)}=\frac{1}{a-1}+o(1) \mbox{ and } D_{1,2}|_{g(x)}=\left.\left(\frac{\Delta_1}{\Delta_2-\Delta_1}\right)\right|_{g(x)}=\frac{1}{a-1}+o(1).
$$
Suppose that the claim is true for any divided difference of order $k-1$. Then $$\left(D_{i_2,\ldots,i_k}-D_{i_1,\ldots,i_{k-1}}\right)|_{g(x)}=d(a)x^{2-k}+o(x^{2-k}).$$ Using Eq. \eqref{e:delta1} we obtain

\begin{align*}
D_{i_1,\ldots,i_k}|_{g(x)}&=\left.\left(\frac{D_{i_2,\ldots,i_k}-D_{i_1,\ldots,i_{k-1}}}{\Delta_{i_k}-\Delta_{i_1}}\right)\right|_{g(x)}
 =\frac{d(a)x^{2-k}+o(x^{2-k})}{(a-1)(a^{i_k}-a^{i_1})x+o(x)} \\
 &  =\frac{d(a)}{(a-1)(a^{i_k}-a^{i_1})}x^{1-k}+ o(x^{1-k}),
\end{align*}
which proves the claim.

Since, from Eq. \eqref{e:delta1} we have $\Delta_0\ldots \Delta_{i-1}|_{g(x)}=\delta_i(a)x^i+o(x^i)$ with $\delta_i(a)\neq 0$, from Eq.  \eqref{e:4} we have that for all $i\geq 2$
\begin{equation}\label{e:5}
\left.\left(\frac{Q_i(\Delta_0,\ldots,\Delta_{i-1})}{R_i(\Delta_0,\ldots,\Delta_i)}\Delta_0\ldots \Delta_{i-1}\right)\right|_{g(x)}=\alpha_i(a)\,x+o(x).
\end{equation} where $\alpha_i(a)=d_i(a)\delta_i(a)$. This ends the first step of the proof.

\emph{Step 2:} Observe that if $f$ is a linear function, so it is the interpolating polynomial of $f^{-1}$, and $g(x)=ax$. Hence
$$
D_{0,\ldots,i}|_{g(x)=ax}=(-1)^i \left.\left(\frac{Q_i(\Delta_0,\ldots,\Delta_{i-1})}{R_i(\Delta_0,\ldots,\Delta_i)}\right)\right|_{g(x)=ax}=0 \mbox{ for all } i\geq 2,
$$ 
hence $d_i(a)=0$ for all $i \geq 2$ and, therefore, when $g(x)=ax+o(x)$
$$
\left.\left(\frac{Q_i(\Delta_0,\ldots,\Delta_{i-1})}{R_i(\Delta_0,\ldots,\Delta_i)
}\Delta_0\ldots \Delta_{i-1}\right)\right|_{g(x)=ax+o(x)}=o(x^2) \mbox{ for all } i\geq 2,
$$ which, together with the result in Eq. \eqref{e:2bis}, gives
$$
K_f^n(g(x))=o(x^2),
$$
with which we conclude the proof.\end{proof}

\begin{proof} [Proof of Theorem~\ref{thm:b}]  For $n\ge 2$ set $\Sigma_n=\R^{n+1}\setminus \{(y_0,y_1,\ldots,y_{n})\in \R^{n+1} \mbox{ such that } y_i\ne y_j\mbox{ for any }  i,j>0,\, i\ne j\}.$
	A Class II method of order~$n$ to approximate a simple zero $s$ of a real function $f$   is given by a ${\mathcal C}^{n}$ function $K:\Sigma_n\longrightarrow \R.$
Directly from Eq. \eqref{ste-gen}, it follows that
this function $K:\Sigma_n\longrightarrow \R$  which defines $K_f^n$ is a rational function. To prove that it has order $n$ instead of using the characterization based on proving that all their derivatives at $s$ of order smaller than $n$ vanish, we will proceed directly from the definition of order. 

	Given some real numbers $c_1,c_2,\ldots ,c_j$ we denote by $\langle c_1,c_2,\ldots ,c_j\rangle$ the smallest interval containing all them.

	The first step is to utilize the error interpolation formula for $f^{-1}(y)$ by using the~$n$ points in Table~\ref{tab1}. We know that
	\[
	f^{-1}(y)-P_{n-1}(y)= \frac{(f^{-1})^{(n)}(\eta_y)}{n!}(y-\Delta_0)(y-\Delta_1)\cdots(y-\Delta_{n-1}),
	\]
	where $\eta_y\in\langle y,\Delta_0,\Delta_1,\ldots,\Delta_{n-1}\rangle,$ see for instance~\cite{SB}. By applying it to $y=0,$ 
		\begin{equation}\label{eq:p1}
	s-K_f^n(x_0)= \frac{(-1)^n(f^{-1})^{(n)}(\eta_0)}{n!}\Delta_0\Delta_1\cdots\Delta_{n-1}.
	\end{equation}
	Let us study with more detail the terms of the right hand side of the above equality.
	
	First notice that
$x_i-s=g(x_{i-1})-g(s)=g'(\rho_{i-1})(x_{i-1}-s),$ with $\rho_{i-1}\in\langle x_{i-1},s\rangle.$ Hence, for all $i\ge 1,$
	\begin{equation}\label{eq:p2}
	x_i-s = g'(\rho_{i-1})g'(\rho_{i-2})\cdots g'(\rho_{0}) (x_0-s),
\end{equation}
with $\rho_{i-1},\rho_{i-2},\dots,\rho_{0}\in \langle x_0,x_1,\ldots, x_{i-1,}s\rangle.$  Moreover, for $i\ge0,$
	\begin{equation}\label{eq:p3}
\Delta_i=x_{i+1}-x_i=g(x_i)-x_i=f(x_i)=f(x_i)-f(s)=f'(\nu_i)(x_i-s),
\end{equation}
with $\nu_i\in\langle x_i,s\rangle.$ By joining~\eqref{eq:p2} and~\eqref{eq:p3} we get that
	\begin{equation}\label{eq:p4}
\Delta_0\Delta_1\cdots\Delta_{n-1}=g'(\rho_{n-1})(g'(\rho_{n-2}))^2\cdots (g'(\rho_{1}))^{n-1} (g'(\rho_{0}))^n\prod_{i=0}^{n-1}f'(\nu_i) (x_0-s)^n.
\end{equation}

 By Lemma \ref{l:ordre2} we know that $K_f^n$ is a  method of order at least 2, hence convergent. Let us  define  $y_0=x_0$ and $\{y_m\}$ as the sequence obtained by iterating this initial condition by~$K_f^n.$  By repeating the above procedure for the step that provides $y_m$ from $y_{m-1}$ and using
 that $\lim_{m\to\infty} y_m=0$   we obtain that
\begin{align*}
&\lim_{m\to\infty} \frac{|y_{m+1}-s|}{|y_m-s|^n} =\lim_{m\to\infty} \frac{|K_f^n(y_m)  -s|}{|y_m-s|^n}\\
&=\lim_{m\to\infty} \frac{|(f^{-1})^{(n)}(\eta_
	{0,m})|}{n!}\Big|g'(\rho_{n-1,m})(g'(\rho_{n-2,m}))^2\cdots (g'(\rho_{1,m}))^{n-1} (g'(\rho_{0,m}))^n\prod_{i=0}^{n-1}f'(\nu_{i,m})\Big|\\&=
\frac{|(f^{-1})^{(n)}(0)|}{n!}|g'(s)|^{n(n+1)/2}|f'(s)|^n\ne0,
\end{align*}
where we have also used~\eqref{eq:p1} and~\eqref{eq:p4} and that  all the unknown intermediate points tend to $s$ or to $0$ when $m$ goes to infinity. Hence the theorem is proved.
\end{proof}

We end this section giving the interpolation polynomials of Table~\ref{tab1} evaluated at $y=0,$ $P_{n-1}(0)$ for small values of $n.$ According Theorem~\ref{thm:b}, these expressions give rise to methods of order $n$ in Class~II.

Following~\eqref{ste-gen} we know that
\begin{equation*}
	P_{n-1}(0)=: K_f^n(x_0)= x_0+\sum_{i=1}^{n-1} \dfrac{Q_{i}(\Delta_0,\Delta_1,\ldots,\Delta_{i})}{R_{i}(\Delta_0,\Delta_1,\ldots,\Delta_{i})}\Delta_0\Delta_1\dots\Delta_{i-1}.
\end{equation*}

Then, writing $Q_i=Q_{i}(\Delta_0,\Delta_1,\ldots,\Delta_{i}) $ and $R_i=R_{i}(\Delta_0,\Delta_1,\ldots,\Delta_{i}),$ after many routine computations we get that
\begin{align*}
	Q_1 &= -\Delta_0,\\
	Q_2&=\Delta_1^2-\Delta_0\Delta_2,\\
	Q_3&={\Delta_{{0}}}^{2}\Delta_{{1}}\Delta_{{3}}-{\Delta_{{0}}}^{2}{\Delta_{
			{2}}}^{2}-\Delta_{{0}}{\Delta_{{1}}}^{2}\Delta_{{3}}+\Delta_{{0}}{
		\Delta_{{2}}}^{3}+\Delta_{{0}}{\Delta_{{2}}}^{2}\Delta_{{3}}\\&\quad-\Delta_{{0
	}}\Delta_{{2}}{\Delta_{{3}}}^{2}+{\Delta_{{1}}}^{3}\Delta_{{2}}-{
		\Delta_{{1}}}^{3}\Delta_{{3}}+{\Delta_{{1}}}^{2}{\Delta_{{3}}}^{2}-
	\Delta_{{1}}{\Delta_{{2}}}^{3},
\end{align*}
\begin{align*}
	Q_4&={\Delta_{{0}}}^{3}{\Delta_{{1}}}^{2}\Delta_{{2}}\Delta_{{4}}-{\Delta_{
			{0}}}^{3}{\Delta_{{1}}}^{2}{\Delta_{{3}}}^{2}-{\Delta_{{0}}}^{3}\Delta
	_{{1}}{\Delta_{{2}}}^{2}\Delta_{{4}}+{\Delta_{{0}}}^{3}\Delta_{{1}}{
		\Delta_{{3}}}^{3}+{\Delta_{{0}}}^{3}\Delta_{{1}}{\Delta_{{3}}}^{2}
	\Delta_{{4}}\\&\quad-{\Delta_{{0}}}^{3}\Delta_{{1}}\Delta_{{3}}{\Delta_{{4}}}^
	{2}+{\Delta_{{0}}}^{3}{\Delta_{{2}}}^{3}\Delta_{{3}}-{\Delta_{{0}}}^{3
	}{\Delta_{{2}}}^{3}\Delta_{{4}}+{\Delta_{{0}}}^{3}{\Delta_{{2}}}^{2}{
		\Delta_{{4}}}^{2}-{\Delta_{{0}}}^{3}\Delta_{{2}}{\Delta_{{3}}}^{3}\\&\quad-{
		\Delta_{{0}}}^{2}{\Delta_{{1}}}^{3}\Delta_{{2}}\Delta_{{4}}+{\Delta_{{0
	}}}^{2}{\Delta_{{1}}}^{3}{\Delta_{{3}}}^{2}+{\Delta_{{0}}}^{2}\Delta_{
		{1}}{\Delta_{{2}}}^{3}\Delta_{{4}}-{\Delta_{{0}}}^{2}\Delta_{{1}}{
		\Delta_{{3}}}^{4}-{\Delta_{{0}}}^{2}\Delta_{{1}}{\Delta_{{3}}}^{3}
	\Delta_{{4}}\\&\quad+{\Delta_{{0}}}^{2}\Delta_{{1}}\Delta_{{3}}{\Delta_{{4}}}^
	{3}-{\Delta_{{0}}}^{2}{\Delta_{{2}}}^{4}\Delta_{{3}}+{\Delta_{{0}}}^{2
	}{\Delta_{{2}}}^{4}\Delta_{{4}}-{\Delta_{{0}}}^{2}{\Delta_{{2}}}^{3}{
		\Delta_{{3}}}^{2}+{\Delta_{{0}}}^{2}{\Delta_{{2}}}^{2}{\Delta_{{3}}}^{
		3}\\&\quad-{\Delta_{{0}}}^{2}{\Delta_{{2}}}^{2}{\Delta_{{4}}}^{3}+{\Delta_{{0}
	}}^{2}\Delta_{{2}}{\Delta_{{3}}}^{4}+\Delta_{{0}}{\Delta_{{1}}}^{3}{
		\Delta_{{2}}}^{2}\Delta_{{4}}-\Delta_{{0}}{\Delta_{{1}}}^{3}{\Delta_{{
				3}}}^{3}-\Delta_{{0}}{\Delta_{{1}}}^{3}{\Delta_{{3}}}^{2}\Delta_{{4}}\\&\quad+
	\Delta_{{0}}{\Delta_{{1}}}^{3}\Delta_{{3}}{\Delta_{{4}}}^{2}-\Delta_{{0
	}}{\Delta_{{1}}}^{2}{\Delta_{{2}}}^{3}\Delta_{{4}}+\Delta_{{0}}{\Delta
		_{{1}}}^{2}{\Delta_{{3}}}^{4}+\Delta_{{0}}{\Delta_{{1}}}^{2}{\Delta_{{
				3}}}^{3}\Delta_{{4}}-\Delta_{{0}}{\Delta_{{1}}}^{2}\Delta_{{3}}{\Delta
		_{{4}}}^{3}\\&\quad+\Delta_{{0}}{\Delta_{{2}}}^{4}{\Delta_{{3}}}^{2}-\Delta_{{0
	}}{\Delta_{{2}}}^{4}{\Delta_{{4}}}^{2}+\Delta_{{0}}{\Delta_{{2}}}^{3}{
		\Delta_{{3}}}^{2}\Delta_{{4}}-\Delta_{{0}}{\Delta_{{2}}}^{3}\Delta_{{3
	}}{\Delta_{{4}}}^{2}+\Delta_{{0}}{\Delta_{{2}}}^{3}{\Delta_{{4}}}^{3}\\&\quad-
	\Delta_{{0}}{\Delta_{{2}}}^{2}{\Delta_{{3}}}^{4}-\Delta_{{0}}{\Delta_{
			{2}}}^{2}{\Delta_{{3}}}^{3}\Delta_{{4}}+\Delta_{{0}}{\Delta_{{2}}}^{2}
	\Delta_{{3}}{\Delta_{{4}}}^{3}+\Delta_{{0}}\Delta_{{2}}{\Delta_{{3}}}^
	{3}{\Delta_{{4}}}^{2}-\Delta_{{0}}\Delta_{{2}}{\Delta_{{3}}}^{2}{
		\Delta_{{4}}}^{3}\\&\quad-{\Delta_{{1}}}^{4}{\Delta_{{2}}}^{2}\Delta_{{3}}+{
		\Delta_{{1}}}^{4}{\Delta_{{2}}}^{2}\Delta_{{4}}+{\Delta_{{1}}}^{4}
	\Delta_{{2}}{\Delta_{{3}}}^{2}-{\Delta_{{1}}}^{4}\Delta_{{2}}{\Delta_{
			{4}}}^{2}-{\Delta_{{1}}}^{4}{\Delta_{{3}}}^{2}\Delta_{{4}}+{\Delta_{{1
	}}}^{4}\Delta_{{3}}{\Delta_{{4}}}^{2}\\&\quad-{\Delta_{{1}}}^{3}{\Delta_{{2}}}
	^{2}{\Delta_{{4}}}^{2}+{\Delta_{{1}}}^{3}\Delta_{{2}}{\Delta_{{4}}}^{3
	}+{\Delta_{{1}}}^{3}{\Delta_{{3}}}^{3}\Delta_{{4}}-{\Delta_{{1}}}^{3}
	\Delta_{{3}}{\Delta_{{4}}}^{3}+{\Delta_{{1}}}^{2}{\Delta_{{2}}}^{4}
	\Delta_{{3}}\\&\quad-{\Delta_{{1}}}^{2}{\Delta_{{2}}}^{4}\Delta_{{4}}+{\Delta_
		{{1}}}^{2}{\Delta_{{2}}}^{3}{\Delta_{{4}}}^{2}-{\Delta_{{1}}}^{2}
	\Delta_{{2}}{\Delta_{{3}}}^{4}-{\Delta_{{1}}}^{2}{\Delta_{{3}}}^{3}{
		\Delta_{{4}}}^{2}\\&\quad+{\Delta_{{1}}}^{2}{\Delta_{{3}}}^{2}{\Delta_{{4}}}^{
		3}-\Delta_{{1}}{\Delta_{{2}}}^{4}{\Delta_{{3}}}^{2}+\Delta_{{1}}{
		\Delta_{{2}}}^{4}{\Delta_{{4}}}^{2}-\Delta_{{1}}{\Delta_{{2}}}^{3}{
		\Delta_{{4}}}^{3}+\Delta_{{1}}{\Delta_{{2}}}^{2}{\Delta_{{3}}}^{4}
\end{align*}
and
\begin{align*}
	R_1&= \Delta_1-\Delta_0,\\
	R_2&=(\Delta_1-\Delta_0) (\Delta_2-\Delta_0) (\Delta_2-\Delta_1), \\
	R_3&=(\Delta_1-\Delta_0) (\Delta_2-\Delta_0) (\Delta_3-\Delta_0) (\Delta_2-\Delta_1) (\Delta_3-\Delta_1) (\Delta_3-\Delta_2),\\
	R_4&=(\Delta_1-\Delta_0) (\Delta_2-\Delta_0) (\Delta_3-\Delta_0)(\Delta_4-\Delta_0) (\Delta_2-\Delta_1) (\Delta_3-\Delta_1)(\Delta_4-\Delta_1)\times\\&\quad\times (\Delta_3-\Delta_2) (\Delta_4-\Delta_2)(\Delta_4-\Delta_3).
\end{align*}

In particular,
\begin{align*}
	K_f^2(x_0)&= x_0-\frac{\Delta_0}{\Delta_1-\Delta_0}\Delta_0=x_0-\frac{(g(x_0)-x_0)^2}{g(g(x_0))-2g(x_0)+x_0},\\
	K_f^3(x_0)&= x_0-\frac{\Delta_0}{\Delta_1-\Delta_0}\Delta_0+ \frac{\Delta_1^2-\Delta_0\Delta_2}{(\Delta_1-\Delta_0) (\Delta_2-\Delta_0) (\Delta_2-\Delta_1)}\Delta_0\Delta_1,
\end{align*}
recovering the Steffensen method ($n=2$) and obtaining a method of order $3$ in Class~II. It is straightforward to write it in terms of $x_0,$ $g(x_0),$ $g(g(x_0)),$ and $g(g(g(x_0))),$ but we skip the details.

\subsection*{Acknowledgments} We thank our friend and colleague Juan Luis Varona for their careful and thoughtful comments and suggestions which help us to improve the presentation of this paper. 
This work is supported by
Ministry of Science and Innovation--State Research Agency of the
Spanish Government through grants PID2022-136613NB-I00   and PID2023-146424NB-I00. It is also supported by the grants 2021-SGR-00113 
and 
2021-SGR-01039 from AGAUR of Generalitat de Catalunya.

\end{document}